\documentclass[reqno]{amsart}
\usepackage{amsmath}
\usepackage{amssymb}
\usepackage{amsfonts}
\usepackage{amsthm}
\usepackage{enumerate}
\usepackage{enumitem}
\usepackage{bbm}
\usepackage[hidelinks]{hyperref}
\usepackage[sort&compress,numbers]{natbib}
\usepackage[left=2cm,right=2cm,bottom=3cm,top=3cm]{geometry}

\newtheorem{theorem}{Theorem}[section]

\newtheorem{proposition}[theorem]{Proposition}
\newtheorem{lemma}[theorem]{Lemma}

\theoremstyle{definition}
\newtheorem*{claim*}{Claim}
\newtheorem{definition}[theorem]{Definition}

\AtBeginDocument{   \def\MR#1{}}

\begin{document}
\title{The Borel complexity of non-Archimedean operator ranges}
\author{Martino Lupini}
\address{Dipartimento di Matematica, Universit\`{a} di Bologna, Piazza di
Porta S. Donato, 5, 40126 Bologna,\ Italy}
\email{martino.lupini@unibo.it}
\urladdr{http://www.lupini.org/}
\thanks{The author was partially supported by the Starting Grant 101077154
\textquotedblleft Definable Algebraic Topology\textquotedblright\ from the
European Research Council, the Gruppo Nazionale per le Strutture Algebriche,
Geometriche e le loro Applicazioni (GNSAGA) of the Istituto Nazionale di
Alta Matematica (INDAM), and the University of Bologna. }
\date{\today }

\begin{abstract}
We completely classify the possible complexity classes of non-closed
operator ranges on separable Banach spaces over a Polish non-Archimedean
non-trivially valued field. These are precisely $\boldsymbol{\Pi }%
_{1+\lambda +n+2}^{0}$ for a countable ordinal $\lambda $ that is either
zero or a limit ordinal, and a finite ordinal $n$. Considering Fr\'echet
spaces produces the additional complexity classes $\boldsymbol{\Pi }%
_{\lambda }^{0}$ for a countable limit ordinal $\lambda $.
\end{abstract}

\setcounter{tocdepth}{1}

\maketitle
\tableofcontents

\section{Introduction}

Operator ranges have been studied in the context of Hilbert spaces, Banach
spaces, and Fr\'echet spaces over real or complex numbers. By definition, an
operator range is simply the image of a continuous linear map. A survey of
results about operator ranges in the context of Hilbert spaces is provided
in \cite{fillmore_operator_1971}. More generally, operator ranges have also
been considered in the context of Banach spaces; see for example \cite%
{jimenez-sevilla_operator_2024,jimenez-sevilla_operator_2023,cross_continuous_1980,fonf_operator_2019}%
.

We restrict our attention to \emph{separable} Banach and Fr\'echet spaces.
Since operator ranges are, in particular, \emph{Borel} subspaces, it is
meaningful to consider their Borel complexity classes, in the sense of \emph{%
Borel complexity theory}; see \cite%
{kechris_classical_1995,gao_invariant_2009,moschovakis_descriptive_2009}. In
the case of operator ranges in the context of Hilbert spaces, it follows
from weak compactness of the unit ball of the Hilbert space that an operator
range is always a countable union of closed sets, i.e., it belongs to the
complexity class $\boldsymbol{\Sigma }_{2}^{0}$. In contrast, in the context
of Banach spaces one can produce operator ranges of arbitrarily high
complexity, as shown by Saint-Raymond in \cite%
{saint-raymond_espaces_1975,saint-raymond_espaces_1976}. The possible
complexity classes that arise in this context have been completely
classified by the author in \cite{lupini_complexity_2025}, both for Banach
spaces and for Fr\'echet spaces.

In this paper, we obtain the natural analogue of such a classification in
the \emph{non-Archimedean} context. We thus consider \emph{separable}
(non-Archimedean) Banach and Fr\'echet spaces over a Polish \emph{%
non-Archimedean non-trivially valued field }$K$. We consider continuous
linear operators on such spaces. As their ranges are always \emph{Borel}
subsets, it is meaningful to discuss their Borel complexity classes, in the
sense of \emph{Borel complexity theory}; see \cite%
{kechris_classical_1995,gao_invariant_2009,moschovakis_descriptive_2009}.

The main goal of this paper is to completely classify the possible Borel
complexity classes of operator ranges on separable Banach or Fr\'echet
spaces over $K$.

\begin{theorem}
\label{Theorem:main}Let $K$ be a Polish non-Archimedean non-trivially valued
field.

\begin{enumerate}
\item The following is a complete list of all the possible Borel complexity
classes of non-closed ranges of operators on separable Banach spaces over $K$%
: $\boldsymbol{\Pi }_{1+\lambda +n+2}^{0}$ for a countable ordinal $\lambda $
that is either zero or a limit ordinal, and a finite ordinal $n$.

\item For non-closed ranges of operators on separable Fr\'echet spaces, the
additional complexity classes $\boldsymbol{\Pi }_{\lambda }^{0}$ for a
nonzero countable limit ordinal $\lambda $ arise.
\end{enumerate}
\end{theorem}

The rest of this paper is organized as follows. Section~\ref%
{Section:complexity} recalls the descriptive and categorical framework.
Section~\ref{Section:non-Archimedean} reviews non-Archimedean Fr\'{e}chet
and Banach spaces and proves the rigidity and upper-bound results. Section~%
\ref{Section:Trees} records the bush topology on well-founded trees. Section~%
\ref{Section:realization} constructs the required realizations and completes
the proof of the classification.

\subsection*{Acknowledgments}

ChatGPT by OpenAI (version GPT-5.6) and Claude by Anthropic (version Opus 5)
have been used in the preparation of this manuscript.

\section{Complexity classes\label{Section:complexity}}

In this section we recall some notions from topology and descriptive set
theory as can be found in \cite{kechris_classical_1995,gao_invariant_2009},
and the theory of groups with a Polish cover as developed in \cite%
{lupini_looking_2024,bergfalk_definable_2024,bergfalk_definable_2024-1,lupini_complexity_2025,casarosa_homological_2026}%
.

\subsection{Ordinal notation}

In this section, we introduce notation for ordinals and \textquotedblleft
fractional ordinals\textquotedblright. We let $\omega $ be the first \emph{%
infinite }ordinal, and $\omega _{1}$ be the first \emph{uncountable }%
ordinal. We let $\omega _{1}[1/2]$ be the set that contains, for every
countable ordinal $\alpha $, both $\alpha $ and $\alpha +1/2$ subject to the
relations%
\begin{equation*}
\alpha <\alpha +\frac{1}{2}<\alpha +1
\end{equation*}%
and%
\begin{equation*}
\beta +(\alpha +\frac{1}{2})=\left( \beta +\alpha \right) +\frac{1}{2}\text{.%
}
\end{equation*}%
We also let $\omega _{1}^{\mathrm{Ph}}$ be the set that contains $\omega
_{1}[1/2]$ and, for every $\alpha <\omega _{1}$, the element $\alpha
+1/2+\varepsilon $ subject to the relations%
\begin{equation*}
\alpha +\frac{1}{2}<\alpha +\frac{1}{2}+\varepsilon <\alpha +1
\end{equation*}%
and%
\begin{equation*}
\beta +(\alpha +\frac{1}{2}+\varepsilon )=\left( \beta +\alpha \right) +%
\frac{1}{2}+\varepsilon \text{\quad and\quad }(\alpha +\frac{1}{2}%
)+\varepsilon =\alpha +\frac{1}{2}+\varepsilon \text{.}
\end{equation*}%
If $\alpha $ is a successor ordinal, we let $\alpha -1$ be its immediate
predecessor, and%
\begin{equation*}
\alpha -\frac{1}{2}=\left( \alpha -1\right) +\frac{1}{2}\text{\quad and\quad 
}\alpha -\frac{1}{2}+\varepsilon =(\alpha -\frac{1}{2})+\varepsilon \text{.}
\end{equation*}%
We also set $\omega _{1}^{\mathbf{\Delta }}\subseteq \omega _{1}^{\mathrm{Ph}%
}$ to be the subset containing $\omega _{1}[1/2]$ together with $\lambda
+1/2+\varepsilon $ whenever $\lambda $ is either zero or a limit ordinal.

\subsection{Complexity classes}

Recall that a Polish space is a second countable topological space that
admits a compatible complete metric. A Polish group is a topological group
whose topology is Polish. A \emph{complexity class }$\Gamma $ is an
assignment $X\mapsto \Gamma \left( X\right) $ of a collection $\Gamma \left(
X\right) $ of Borel subsets of $X$ to each Polish space $X$, in such a way
that a continuous function $f:X\rightarrow Y$ induces a function $\Gamma
\left( Y\right) \rightarrow \Gamma \left( X\right) $ by taking preimages. We
will consider the complexity classes $\boldsymbol{\Pi }_{\alpha }^{0}$ and $%
\boldsymbol{\Sigma }_{\alpha }^{0}$ for $\alpha <\omega _{1}$. For $\alpha
<\omega _{1}$ we will also consider the complexity class $D(\boldsymbol{\Pi }%
_{\alpha }^{0})$ comprising the sets that can be written as the intersection
of an element of $\boldsymbol{\Pi }_{\alpha }^{0}$ and an element of $%
\boldsymbol{\Sigma }_{\alpha }^{0}$.

The \emph{dual class }$\check{\Gamma}$ of a complexity class $\Gamma $
contains the \emph{complements }of the elements of $\Gamma $. A complexity
class is not \emph{self-dual} if it is different from its dual class. In
this case, for a Polish space $X$ and a Borel subset $A$, one says that $%
\Gamma $ is the complexity class of $A$ if $A\in \Gamma \left( X\right) $
and $A\notin \check{\Gamma}\left( X\right) $.

We define an increasing sequence $\left( \Gamma _{\alpha }\right) _{\alpha
\in \omega _{1}^{\mathrm{Ph}}}$ of complexity classes as follows: For $%
\lambda <\omega _{1}$ either zero or limit, $n<\omega $, and $i\in \left\{
0,1/2,1/2+\varepsilon \right\} $, define:%
\begin{equation*}
\Gamma _{\lambda +n+i}:=\left\{ 
\begin{array}{ll}
\boldsymbol{\Pi }_{1+\lambda }^{0} & \text{for }n=0\text{ and }i=0\text{;}
\\ 
\boldsymbol{\Sigma }_{1+\lambda +1}^{0} & \text{for }n=0\text{ and }i=1/2%
\text{;} \\ 
D(\boldsymbol{\Pi }_{1+\lambda +1}^{0}) & \text{for }n=0\text{ and }%
i=1/2+\varepsilon \text{;} \\ 
D(\boldsymbol{\Pi }_{1+\lambda +n+1}^{0}) & \text{for }n\geq 1\text{ and }%
i\in \left\{ 1/2,1/2+\varepsilon \right\} \text{;} \\ 
\boldsymbol{\Pi }_{1+\lambda +n+1}^{0} & \text{for }n\geq 1\text{ and }i=0%
\text{;}%
\end{array}%
\right.
\end{equation*}%
For a Borel subset $A$ of a Polish space $X$, we define the \emph{Borel
class }of $A$ in $X$ to be the least $\alpha \in \omega _{1}^{\mathrm{Ph}}$
such that $A\in \Gamma _{\alpha }\left( X\right) $. Notice that for every
successor ordinal $\sigma $, the classes $\Gamma _{\sigma +1/2}$ and $\Gamma
_{\sigma +1/2+\varepsilon }$ are equal. Thus, by definition the Borel class
of $A$ in $X$ must be in $\omega _{1}^{\mathbf{\Delta }}$.

\subsection{Homogeneous spaces with a Polish cover}

Let $G$ be a Polish group. A \emph{Polish subgroup }of $G$ (also called 
\emph{Polishable }subgroup) is a subgroup $N$ of $G$ that is also a Polish
group with respect to a (necessarily unique) Polish group topology that
makes the inclusion $N\rightarrow G$ continuous. A homogeneous space with a
Polish cover is a pointed space of the form $G/N$ where $G$ is a Polish
group and $N$ is a Polish subgroup of $G$.\ A homogeneous subspace with a
Polish cover of $G/N$ is a subspace of the form $H/N$ where $H$ is a Polish
subgroup of $G$ containing $N$.\ In particular, $\left\{ \ast \right\} =N/N$
is a homogeneous subspace with a Polish cover. We define the Borel class of $%
H/N$ in $G/N$ to be the Borel class of $H$ in $G$. It is proved in \cite%
{lupini_complexity_2025} building on \cite%
{solecki_polish_1999,farah_borel_2006} that for every $\alpha \in \omega
_{1} $ there exists a \emph{least }homogeneous subspace with a Polish cover 
\textrm{Ph}$^{\alpha }\left( G/N\right) =G_{\alpha }/N$ of $G/N$ of Borel
class at most $\alpha $.

We define the \emph{Borel length }$\ell _{\boldsymbol{\Delta }}\left(
G/N\right) \in \omega _{1}^{\mathrm{Ph}}$ of $G/N$ to be the Borel class of $%
\left\{ \ast \right\} $ in $G/N$ or, equivalently, the Borel class of $N$ in 
$G$. One defines the \emph{phantom length }$\ell _{\mathrm{Ph}}\left(
G/N\right) \in \omega _{1}^{\mathrm{Ph}}$ of $G/N$ as follows:

\begin{definition}
Let $G/N$ be a homogeneous space with a Polish cover. Set $\mathrm{Ph}%
^{\alpha }\left( G/N\right) :=G_{\alpha }/N$. For $\alpha <\omega _{1}$
define:

\begin{itemize}
\item $\ell _{\mathrm{Ph}}\left( G/N\right) \leq \alpha \Leftrightarrow 
\mathrm{Ph}^{\alpha }\left( G/N\right) =\left\{ \ast \right\} $;

\item $\ell _{\mathrm{Ph}}\left( G/N\right) \leq \alpha +1/2\Leftrightarrow
\left\{ \ast \right\} \in \boldsymbol{\Sigma }_{2}^{0}(\mathrm{Ph}^{\alpha
}\left( G/N\right) )$;

\item $\ell _{\mathrm{Ph}}\left( G/N\right) \leq \alpha +1/2+\varepsilon
\Leftrightarrow \left\{ \ast \right\} \in D(\boldsymbol{\Pi }_{2}^{0})(%
\mathrm{Ph}^{\alpha }\left( G/N\right) )$.
\end{itemize}
\end{definition}

By \cite[Theorem 6.1]{lupini_complexity_2025}, the phantom length of $G/N$
completely determines the Borel class of $\left\{ \ast \right\} $ in $G/N$:

\begin{theorem}
Let $G/N$ be a homogeneous space with a Polish cover. Suppose that $\alpha
\in \omega _{1}^{\mathrm{Ph}}$. Then%
\begin{equation*}
\ell _{\boldsymbol{\Delta }}\left( G/N\right) \leq \alpha \Leftrightarrow
\ell _{\mathrm{Ph}}\left( G/N\right) \leq \alpha \text{.}
\end{equation*}
\end{theorem}

One defines a group with a Polish cover to be a homogeneous space with a
Polish cover $G/N$ in which $N$ is normal in $G$. In this case, a subgroup
with a Polish cover is a homogeneous subspace with a Polish cover $H/N$ such
that $H$ is normal in $G$.\ When $G/N$ is a group with a Polish cover, $%
\mathrm{Ph}^{\alpha }\left( G/N\right) $ is a subgroup with a Polish cover
for every $\alpha <\omega _{1}$.

\subsection{Non-Archimedean groups}

We recall the definition of non-Archimedean Polish group. A Polish group $G$
is \emph{non-Archimedean} if it has a basis of identity neighborhoods
consisting of (necessarily open) subgroups. A homogeneous space with a
non-Archimedean Polish cover is a homogeneous space with a Polish cover $G/N$
where both $G$ and $N$ are non-Archimedean. In this case, $\mathrm{Ph}%
^{\alpha }\left( G/N\right) $ is also a homogeneous space with a
non-Archimedean Polish cover for every $\alpha <\omega _{1}$.

If $G/N$ is a homogeneous space with a Polish cover, and $N$ is
non-Archimedean, then the \emph{phantom length} of $G/N$ necessarily belongs
to $\omega _{1}[1/2]$ by the main results of \cite{hjorth_borel_1998} and 
\cite[Theorem 3.3]{lupini_complexity_2025}; see \cite[Theorem 6.1]%
{lupini_complexity_2025}. Thus, in this case we have:

\begin{theorem}
Let $G/N$ be a homogeneous space with a Polish cover. If $N$ is a
non-Archimedean Polish group, then%
\begin{equation*}
\ell _{\boldsymbol{\Delta }}\left( G/N\right) =\ell _{\mathrm{Ph}}\left(
G/N\right) \in \omega _{1}[1/2]\text{.}
\end{equation*}
\end{theorem}

\subsection{Homomorphisms}

Suppose that $G/N$ and $H/M$ are homogeneous spaces with a Polish cover. A
function $f:G/N\rightarrow H/M$ is \emph{Borel-definable} if it admits a 
\emph{Borel lift} $\varphi :G\rightarrow H$. By definition, this means that $%
\varphi $ is a Borel function such that, for every $g\in G$, $\varphi \left(
g\right) M=f\left( gN\right) $. It is shown in \cite%
{bergfalk_definable_2024-1} that if $f$ is a Borel-definable bijection, then
its inverse is also Borel-definable.

It follows from \cite[Theorem 3.3]{lupini_complexity_2025} that if $%
f:G/N\rightarrow H/M$ is an \emph{injective }Borel-definable function, then
for $\alpha \in \omega _{1}\setminus \left\{ 1/2\right\} $:%
\begin{equation*}
\ell _{\boldsymbol{\Delta }}\left( G/M\right) =\alpha \Rightarrow \ell _{%
\boldsymbol{\Delta }}\left( G/N\right) \leq \alpha \text{.}
\end{equation*}

Thus, if $f$ is a \emph{bijective }Borel-definable function, for $\alpha \in
\omega _{1}\setminus \left\{ 1/2,1/2+\varepsilon \right\} $:

\begin{itemize}
\item $\ell _{\boldsymbol{\Delta }}\left( H/M\right) =\alpha \Leftrightarrow
\ell _{\boldsymbol{\Delta }}\left( G/N\right) =\alpha $;

\item $\ell _{\boldsymbol{\Delta }}\left( H/M\right) \in \left\{
1/2,1/2+\varepsilon \right\} \Leftrightarrow \ell _{\boldsymbol{\Delta }%
}\left( G/N\right) \in \left\{ 1/2,1/2+\varepsilon \right\} $.
\end{itemize}

When $N$ and $M$ are non-Archimedean Polish groups, the phantom lengths $%
\alpha +1/2+\varepsilon $ for some countable ordinal $\alpha $ cannot occur.
Thus the Borel length and phantom length coincide, and we have the following:

\begin{proposition}
Let $G/N$ and $H/M$ be homogeneous spaces with a Polish cover, where $N$ and 
$M$ are non-Archimedean Polish groups. Let $\varphi :G/N\rightarrow H/M$ be
a Borel-definable function. Then:

\begin{enumerate}
\item if $\varphi $ is injective, then $\ell _{\mathrm{Ph}}\left( G/N\right)
\leq \ell _{\mathrm{Ph}}\left( H/M\right) $;

\item if $\varphi $ is bijective, then $\ell _{\mathrm{Ph}}\left( G/N\right)
=\ell _{\mathrm{Ph}}\left( H/M\right) $.
\end{enumerate}
\end{proposition}

\subsection{The left heart}

Let $R$ be a Polish ring. Denote by $\mathbf{PMod}\left( R\right) $ the
category of \emph{Polish }topological $R$-modules. Then $\mathbf{PMod}\left(
R\right) $ is a \emph{quasi-abelian }category in the sense of \cite%
{schneiders_quasi-abelian_1999}. As such, it admits a \emph{left heart }$%
\mathrm{LH}\left( \mathbf{PMod}\left( R\right) \right) $, which is an
abelian category uniquely characterized by a suitable universal property;
see \cite{schneiders_quasi-abelian_1999}. Following \cite%
{lupini_looking_2024}, one can describe $\mathrm{LH}\left( \mathbf{PMod}%
\left( R\right) \right) $ as the category of $R$-modules with a Polish
cover. These are precisely the groups with a Polish cover $G/N$ where both $%
N $ and $G$ are Polish $R$-modules, and the inclusion $N\rightarrow G$ is an 
$R $-module homomorphism (an $R$-homomorphism). The morphisms $%
G/N\rightarrow H/M$ in $\mathrm{LH}\left( \mathbf{PMod}\left( R\right)
\right) $ are precisely the \emph{Borel-definable }$R$-homomorphisms, i.e.,
the Borel-definable group homomorphisms that are also $R$-module
homomorphisms.

Let $\mathcal{B}$ be a full subcategory of $\mathbf{PMod}\left( R\right) $
that is closed under taking closed submodules and quotients by closed
submodules. Define a module with a cover in $\mathcal{B}$ to be a module
with a Polish cover of the form $G/N$ where both $G$ and $N$ are isomorphic
to objects of $\mathcal{B}$.

If $\mathcal{B}$ is also a quasi-abelian category, the inclusion $\mathcal{B}%
\rightarrow \mathbf{PMod}\left( R\right) $ preserves finite limits and
finite colimits, and $\mathcal{B}$ is essentially closed under extensions in 
$\mathbf{PMod}\left( R\right) $, then its left heart $\mathrm{LH}\left( 
\mathcal{B}\right) $ is (equivalent to) the full subcategory of $\mathrm{LH}%
\left( \mathbf{PMod}\left( R\right) \right) $ spanned by the modules with a
cover in $\mathcal{B}$; see \cite[Section 6]{lupini_looking_2024}.

For convenience of notation, we introduce the following terminology.

\begin{definition}
\label{Definition:phantom-spectrum}Let $R$ be a Polish ring. Let $\mathcal{B}
$ be a full subcategory of the category $\mathbf{PMod}\left( R\right) $ of
Polish modules over $R$ that is closed under taking closed submodules and
quotients by closed submodules. The \emph{phantom spectrum} $\sigma _{%
\mathrm{Ph}}\left( \mathcal{B}\right) $ of $\mathcal{B}$ is the subset of $%
\omega _{1}^{\mathrm{Ph}}$ comprising the phantom lengths of modules with a
cover in $\mathcal{B}$.
\end{definition}

\section{Non-Archimedean Banach and Fr\'echet spaces\label%
{Section:non-Archimedean}}

We recall some fundamental notions from the theory of valued fields and
non-Archimedean functional analysis, as can be found in \cite%
{bosch_non-archimedean_1984,schneider_nonarchimedean_2002,van_rooij_non-archimedean_1978,perez-garcia_locally_2010,groenewegen_spaces_2016}%
.

\subsection{Non-Archimedean Fr\'echet spaces}

Throughout the paper, $K$ is a field with a complete non-Archimedean
absolute value $|\cdot |$ that renders $K$ a Polish space. All vector spaces
are assumed to be over $K$, and all linear maps are assumed to be $K$%
-linear. The valuation is \emph{nontrivial}, meaning that $\left\vert \pi
\right\vert \neq 1$ for some nonzero $\pi \in K$. After replacing $\pi $ by
its inverse if necessary, we can assume that $0<\left\vert \pi \right\vert
<1 $. One then sets 
\begin{equation*}
\mathcal{O}_{K}:=\{a\in K:\left\vert a\right\vert \leq 1\}
\end{equation*}%
which is the valuation ring associated with the given valuation. We consider 
$K$ as a topological field with respect to the topology induced by this
absolute value. Notice that%
\begin{equation*}
K=\bigcup_{n\in \mathbb{N}}\pi ^{-n}\mathcal{O}_{K}\text{.}
\end{equation*}

Let $X$ be a \emph{Polish} topological vector space. Notice that, in
particular, $X$ is a topological $\mathcal{O}_{K}$-module. A subset $C$ of $%
X $ is called \emph{absolutely convex }if it is an $\mathcal{O}_{K}$%
-submodule of $X$, and an \emph{open lattice }if it is absolutely convex and
open. A (non-Archimedean) separable Fr\'{e}chet space is a Polish
topological vector space that admits a zero neighborhood basis consisting of
open lattices. Equivalently, $X$ is a Polish topological vector space whose
topology is induced by a countable family of (non-Archimedean) seminorms.

Let $X$ be a separable Fr\'{e}chet space. A subset $A$ of $X$ is \emph{%
bounded }if for every zero neighborhood $U$ of $X$ there exists a zero
neighborhood $V$ of $K$ such that $V\cdot A\subseteq U$; see \cite%
{weiss_boundedness_1956}. This is equivalent to the assertion that for every
zero neighborhood $U$ of $X$ there exists $\lambda \in K$ such that $%
A\subseteq \lambda U$ \cite[Definition 3.6.1]{perez-garcia_locally_2010}. A
separable Banach space is a separable Fr\'{e}chet space that is \emph{%
locally bounded}, i.e., it admits a bounded zero neighborhood. This is
equivalent to the assertion that $X$ is a Polish topological vector space
whose topology is induced by a (non-Archimedean) norm; see \cite[Definition
2.1.1]{perez-garcia_locally_2010}, \cite[Theorem 3.6.2]%
{perez-garcia_locally_2010}.

In what follows, all Fr\'{e}chet and Banach spaces are assumed to be \emph{%
separable}.

\subsection{Fr\'{e}chet subspaces}

Suppose that $X$ is a Fr\'echet space. A Fr\'echet subspace of $X$ is a
vector subspace $Y$ that is endowed with a (necessarily unique) Fr\'{e}chet
space topology that makes the inclusion $Y\rightarrow X$ continuous.
Likewise, $Y$ is a Banach subspace of $X$ if it is endowed with a
(necessarily unique) Banach space topology that makes the inclusion $%
Y\rightarrow X$ continuous.

As in \cite{lupini_looking_2024}, one defines a space with a Fr\'echet cover
to be $X/N$ where $X$ is a Fr\'echet space and $N\subseteq X$ is a Fr\'echet
subspace. A subspace with a Fr\'echet cover of $X/N$ is a subspace $Y/N$
where $Y$ is a Fr\'echet subspace of $X$ containing $N$. For $\alpha <\omega
_{1}$, $\mathrm{Ph}^{\alpha }\left( X/N\right) $ is a subspace with a
Fr\'echet cover. A \emph{phantom Fr\'echet space }is a space with a
Fr\'echet cover $X/N$ with $N$ dense in $X$.

Analogous definitions apply to Banach spaces rather than Fr\'echet spaces.
If $X/N$ is a space with a Fr\'echet cover where $N$ is a Banach space, then
for every $\alpha <\omega _{1}$, $\mathrm{Ph}^{\alpha +1}\left( X/N\right) $
is a space with a Banach cover.

\subsection{Operator ranges}

Let $X,Y$ be Fr\'{e}chet spaces. Let $T:X\rightarrow Y$ be a continuous
linear map. Then its image or range $\mathrm{Ran}\left( T\right) \subseteq Y$
is a Fr\'{e}chet subspace of $Y$, isomorphic to the quotient of $X$ by the
closed subspace $\mathrm{Ker}\left( T\right) $. Conversely, every Fr\'{e}%
chet subspace $Z$ of $Y$ is the range of the continuous linear operator $%
Z\rightarrow Y$ given by the inclusion. Thus, the problem of classifying the
possible Borel complexity classes of ranges of continuous linear maps
between Fr\'{e}chet spaces (operator ranges) is equivalent to the problem of
classifying the possible \emph{phantom lengths }of spaces with a Fr\'{e}chet
cover, or even just \emph{phantom Fr\'{e}chet spaces}.

In terms of Definition \ref{Definition:phantom-spectrum}, one can
reformulate the statement of Theorem \ref{Theorem:main} as follows. Let $K$
be a Polish non-Archimedean non-trivially valued field. We denote by $%
\mathbf{Ban}\left( K\right) $ the full subcategory of $\mathbf{PMod}\left(
K\right) $ comprising the separable Banach spaces. Likewise, $\mathbf{Fre}%
\left( K\right) $ is the full subcategory of $\mathbf{PMod}\left( K\right) $
comprising the separable Fr\'{e}chet spaces. Notice that $\mathbf{Fre}\left(
K\right) $ consists precisely of the limits of towers of separable Banach
spaces with continuous linear maps with dense image as bonding maps. Indeed,
for a separable Fr\'{e}chet space $X$, choose an increasing sequence $\left(
p_{n}\right) $ of seminorms defining its topology, and let $X_{n}$ be the
Banach completion of $X/\mathrm{Ker}\left( p_{n}\right) $. The canonical map 
$X\rightarrow \varprojlim_{n}X_{n}$ is a topological isomorphism.
Conversely, a projective limit of a countable tower of separable Banach
spaces is a separable Fr\'{e}chet space; see \cite[Definition 3.4.29 and
Corollary 3.5.7]{perez-garcia_locally_2010}.

\begin{theorem}
\label{Theorem:spectra}Let $K$ be a Polish non-Archimedean non-trivially
valued field. Then:

\begin{enumerate}
\item the phantom spectrum $\sigma _{\mathrm{Ph}}\left( \mathbf{Ban}\left(
K\right) \right) $ of $\mathbf{Ban}\left( K\right) $ is the set of countable
zero or successor ordinals;

\item the phantom spectrum $\sigma _{\mathrm{Ph}}\left( \mathbf{Fre}\left(
K\right) \right) $ of $\mathbf{Fre}\left( K\right) $ is the set of countable
ordinals.
\end{enumerate}
\end{theorem}

\subsection{Rigidity}

We begin by studying the value $1/2$ for the phantom length of phantom
Fr\'echet spaces and show that it cannot occur.

For an $\mathcal{O}_{K}$-submodule $C$ of a $K$-vector space, define 
\begin{equation*}
KC=\left\{ \lambda x:x\in C\text{ and }\lambda \in K\right\} \text{,}
\end{equation*}%
which is the $K$-linear span of $C$. Powers of $\pi $ absorb every scalar,
and therefore 
\begin{equation}
KC=\bigcup_{j<\omega }\pi ^{-j}C.  \label{eq:linear-hull}
\end{equation}

\begin{proposition}
\label{Proposition:closed-lattice} Let $Y$ be a non-Archimedean Fr\'{e}chet
space over $K$, and let $C\subseteq Y$ be a closed $\mathcal{O}_{K}$%
-submodule. If $\pi C$ is open in $C$, then $KC$ is closed in $Y$.
\end{proposition}

\begin{proof}
We claim that $C$ is \emph{open} in $KC$ for the relative topology inherited
from $Y$. Suppose that this is not the case. Thus, $KC\setminus C$ is not
closed in $KC$. Thus, there exists a sequence $\left( z_{n}\right) $ in $%
KC\setminus C$ that converges to $y\in C$. After replacing $z_{n}$ with $%
z_{n}-y$ we can assume that $y=0$. By \eqref{eq:linear-hull}, for every $%
n\in \mathbb{N}$ there is a \emph{least} $j_{n}<\omega $ such that $\pi
^{j_{n}}z_{n}\in C$. Since $z_{n}\notin C$, we must have $j_{n}\geq 1$. For $%
n\in \mathbb{N}$, define%
\begin{equation*}
c_{n}=\pi ^{j_{n}}z_{n}\in C\text{.}
\end{equation*}
Minimality of $j_{n}$ shows that $c_{n}$ does not belong to $\pi C$. We
claim that $\left( c_{n}\right) $ converges to $0$ in $C$ with respect to
the relative topology inherited from $Y$. Indeed, for each continuous
seminorm $p$ on $Y$, one has%
\begin{equation*}
p\left( c_{n}\right) =\left\vert \pi \right\vert ^{j_{n}}p\left(
z_{n}\right) \leq p\left( z_{n}\right) \text{.}
\end{equation*}

Thus, 
\begin{equation*}
\mathrm{lim}_{n}p\left( c_{n}\right) \leq \mathrm{lim}_{n}p\left(
z_{n}\right) =0\text{.}
\end{equation*}%
As this holds for every continuous seminorm $p$ on $Y$, and $Y$ is a
Fr\'echet space, we must have $\mathrm{lim}_{n}c_{n}=0$. Since by hypothesis 
$\pi C$ is open in $C$, and $0\in \pi C$, this implies that, for all but
finitely many $n\in \mathbb{N}$, $c_{n}\in \pi C$. This is a contradiction.

We therefore conclude that $C$ is open in $KC$. Suppose now that $y\in 
\overline{KC}^{Y}$. Then there exists a sequence $\left( x_{n}\right) $ in $%
KC$ that converges to $y$ in $Y$. Since $C$ is open in $KC$ with respect to
the subspace topology inherited from $Y$, there is a zero neighborhood $V$
in $Y$ such that $V\cap KC\subseteq C$. Since the sequence $\left(
x_{n}\right) $ converges in $Y$, it is Cauchy. Thus, there exists $n_{0}\in 
\mathbb{N}$ such that for all $n,m\geq n_{0}$, $x_{n}-x_{m}\in V$. Since $KC$
is a subspace of $Y$ and $x_{n}\in KC$ for every $n\in \mathbb{N}$, we also
have $x_{n}-x_{m}\in KC$ for every $n,m\geq n_{0}$. Thus, for all $n,m\geq
n_{0}$ we have%
\begin{equation*}
x_{n}-x_{m}\in V\cap KC\subseteq C\text{.}
\end{equation*}%
In particular,%
\begin{equation*}
x_{n}\in C+x_{n_{0}}
\end{equation*}%
for all $n\geq n_{0}$. Since $C$ is closed in $Y$ by hypothesis, we have%
\begin{equation*}
y=\mathrm{lim}_{n}x_{n}\in C+x_{n_{0}}\text{.}
\end{equation*}%
Thus, $y\in C+x_{n_{0}}\subseteq C+KC\subseteq KC$. As $y$ is an arbitrary
element of the closure of $KC$ in $Y$, this concludes the proof that $KC$ is
closed in $Y$.
\end{proof}

\begin{lemma}
\label{Lemma:length-less-than-1}Let $X/N$ be a phantom Fr\'{e}chet space. If 
$X/N$ has phantom length less than $1$, then it has phantom length $0$.
\end{lemma}

\begin{proof}
If $X/N$ has phantom length less than $1$, then, since $N$ is
non-Archimedean, the phantom length of $X/N$ is at most $1/2$. Then by the
Baire Category Theorem, $N$ has a zero neighborhood $B$ that is closed in $X$%
. One can take $B$ to be an \emph{open lattice }in $N$.

We claim that $\pi B$ is open in $B$ for the topology inherited from $X$.
Multiplication by $\pi $ is a homeomorphism of $X$, hence $\pi B$ is closed
in $X$. It is also open in $N$ for the given Fr\'{e}chet topology of $N$,
because $B$ is open in $N$. Since $N$ is separable in that topology, the
discrete quotient $N/\pi B$ is countable. In particular, $B/\pi B$ is
countable, and $\pi B$ is open in $B$.

Proposition \ref{Proposition:closed-lattice} now implies that $KB$ is closed
in $X$. We show that $KB=N$. Since $N$ is a vector subspace of $X$ and $%
B\subseteq N$, one has $KB\subseteq N$. Conversely, if $x\in N$, then $\pi
^{n}x\rightarrow 0$ in $N$. As $B$ is open in $N$, there exists $n\in 
\mathbb{N}$ such that $\pi ^{n}x\in B$. Hence $x\in \pi ^{-n}B\subseteq KB$.
Thus $N=KB$ is closed in $X$. Finally, $N$ is dense in $X$ because $X/N$ is
phantom, and therefore $N=X$ and $X/N=0$.
\end{proof}

From Lemma \ref{Lemma:length-less-than-1} we obtain immediately the
following restriction on the possible values of the phantom length of a
space with a Fr\'echet cover.

\begin{lemma}
\label{Lemma:upper-bound}Let $X/N$ be a space with a Fr\'echet cover. Then:

\begin{enumerate}
\item the phantom length $\ell _{\mathrm{Ph}}\left( X/N\right) $ is in $%
\omega _{1}\subseteq \omega _{1}^{\mathrm{Ph}}$;

\item if $X/N$ is a space with a Banach cover, then $\ell _{\mathrm{Ph}%
}\left( X/N\right) $ is either $0$ or a successor ordinal.
\end{enumerate}
\end{lemma}

\begin{proof}
(1) By \cite[Theorem 6.1]{lupini_complexity_2025}, as remarked above, the
phantom length belongs to $\omega _{1}[1/2]$. The conclusion thus follows
from Lemma \ref{Lemma:length-less-than-1} and the definition of phantom
length.

(2) This is the non-Archimedean version of the argument in \cite[Proposition
10.3]{lupini_complexity_2025}. By (1), it is enough to rule out a nonzero
limit ordinal. Write 
\begin{equation*}
\mathrm{Ph}^{\alpha }(X/N)=X_{\alpha }/N
\end{equation*}%
for $\alpha <\omega _{1}$. Suppose that $N=X_{\lambda }$ for some limit
ordinal $\lambda $. Since by assumption $N$ is a Banach space, it has a
bounded zero neighborhood $U$. Since the Polish topology on $X_{\lambda }$
is the projective limit topology induced from $X_{\alpha }$ for $\alpha
<\lambda $, there exists $\beta <\lambda $ and a zero neighborhood $W$ in $%
X_{\beta }$ such that $W\cap N\subseteq U$. Since $U$ is bounded, $\left(
\pi ^{n}U\right) _{n\in \mathbb{N}}$ is a basis of zero neighborhoods in $N$%
. Since $W\cap N\subseteq U$ we have%
\begin{equation*}
\pi ^{n}W\cap N\subseteq \pi ^{n}U
\end{equation*}%
for every $n\in \mathbb{N}$. Since $\pi ^{n}W$ is a zero neighborhood in $%
X_{\beta }$ for every $n\in \mathbb{N}$, this proves that $N$ is endowed
with the subspace topology inherited from $X_{\beta }$. Since $N$ is dense
in $X_{\beta }$, this forces $N=X_{\beta }$. As this holds for an arbitrary
countable limit ordinal $\lambda $, we conclude that the phantom length of $%
X/N$ is not a limit ordinal.
\end{proof}

\section{Bushes and their spaces\label{Section:Trees}}

We now recall some notions pertaining to (order-theoretic) trees and bushes,
as can be found in \cite%
{kechris_classical_1995,schroder_ordered_2016,rudeanu_sets_2012,harzheim_ordered_2005,davey_introduction_1990}%
.

\subsection{Well-founded relations}

A relation $R$ on a nonempty set $X$ is \emph{well-founded} if every
nonempty subset of $X$ has an $R$-minimal element. In this case, one can
define the $R$-rank $\mathrm{rk}_{R}\left( x\right) $ of an element $x$ of $%
X $ by well-founded recursion. Thus, $\mathrm{rk}_{R}\left( x\right) =0$ if
and only if $x$ is $R$-minimal, and%
\begin{equation*}
\mathrm{rk}_{R}\left( x\right) =\mathrm{sup}\left\{ \mathrm{rk}_{R}\left(
y\right) +1:yRx\right\}
\end{equation*}%
otherwise. Then%
\begin{equation*}
\mathrm{rk}\left( R\right) :=\mathrm{sup}\left\{ \mathrm{rk}_{R}\left(
x\right) +1:x\in X\right\} \text{.}
\end{equation*}%
In what follows, we will mainly consider well-founded relations that are
(strict) partial orders, i.e., irreflexive and transitive.

\subsection{Order-theoretic trees}

An ordered set $P$ is \emph{directed} if for every $x,y\in P$ there exists $%
z\in P$ such that $x\leq z$ and $y\leq z$. A \emph{forest} is an ordered set 
$T$ such that, for every $s\in T$, the set%
\begin{equation*}
\uparrow s:=\left\{ t\in T:s\leq t\right\}
\end{equation*}%
of elements above $s$ is a \emph{finite linear order}. For $s\in T$ we also
set%
\begin{equation*}
\downarrow s:=\left\{ t\in T:t\leq s\right\} \text{.}
\end{equation*}%
A countable (rooted) \emph{tree} is a countable forest that is \emph{directed%
}. This is equivalent to the assertion that $T$ has a largest element,
called the \emph{root}. A tree is \emph{well-founded} if it is well-founded
as an ordered set, in which case its rank is defined as above. The leaves of
a well-founded tree are its minimal elements, i.e., the nodes of rank $0$. A
countable tree is \emph{infinitely branching} if every node that is not a
leaf has infinitely many immediate predecessors (children). It is easily
seen that any \emph{forest} admits a partition into pairwise-incomparable
trees.

In a forest $F$, we define the children of $x\in F$ to be its immediate
predecessors (if any). We let $\mathrm{Ch}\left( x\right) $ be the (possibly
empty) set of children of $x$.

\subsection{The bush topology}

We define:

\begin{itemize}
\item a \emph{bushland }to be a countable well-founded forest such that
every node has either none or infinitely many immediate predecessors;

\item a \emph{bush }to be a bushland that is also a tree.
\end{itemize}

\begin{definition}
\label{Definition:bush-topology}Let $B$ be a bushland. A subset $O$ of $B$
is \emph{open }if for every $x\in O$, $O$ contains the set $\downarrow t$
for all but finitely many children $t$ of $x$.
\end{definition}

The condition in Definition \ref{Definition:bush-topology} defines a
topology. Observe also that any downward-closed set is open. We now
establish some properties of the bush topology. Given a sequence $\left(
b_{n}\right) $ in $B$, we define an \emph{eventual upper bound} of $\left(
b_{n}\right) $ to be an upper bound of one of its tails. We say $%
b=\limsup_{n}b_{n}$ if and only if $b$ is the \emph{least} element of the
set of eventual upper bounds. For $\alpha <\omega _{1}$ we define:

\begin{itemize}
\item $\delta _{\alpha }B$ to be the set of nodes of $B$ of rank less than $%
\alpha $;

\item $\sigma _{\alpha }B$ to be the set of nodes of $B$ of rank equal to $%
\alpha $;

\item $\partial _{\alpha }B$ to be the set of nodes of $B$ of rank greater
than or equal to $\alpha $.
\end{itemize}

Recall that a topological space is \emph{scattered }if every nonempty
subspace has an isolated point; see \cite[Exercise 30E]{willard_general_2004}
and \cite[1.7.10]{engelking_general_1989}.

\begin{theorem}
\label{Theorem:bush-topology}Let $B$ be a bushland and $\alpha <\omega _{1}$%
. Then:

\begin{enumerate}
\item every node of $B$ has a countable neighborhood basis;

\item $B$ has a countable basis of compact open sets;

\item $B$ is compact if and only if it has finitely many roots;

\item $B$ is Hausdorff;

\item a sequence $\left( b_{n}\right) $ in $B$ converges to $b$ if and only
if $b=\limsup_{k}b_{n_{k}}$ for each subsequence $\left( b_{n_{k}}\right) $;

\item $\sigma _{\alpha }B$ is discrete;

\item $B$ is scattered.
\end{enumerate}
\end{theorem}

\begin{proof}
For $x\in B$, write $\mathrm{Ch}(x)$ for the set of children of $x$. If $F$
is a finite subset of $\mathrm{Ch}(x)$, put 
\begin{equation*}
U(x,F):=\{x\}\cup \bigcup_{t\in \mathrm{Ch}(x)\setminus F}\downarrow t.
\end{equation*}%
The set $U(x,F)$ is open. Moreover, if $O$ is an open set containing $x$,
then the definition of the bush topology yields a finite set $F\subseteq 
\mathrm{Ch}(x)$ such that $U(x,F)\subseteq O$. Since $B$ is countable, there
are only countably many pairs $(x,F)$ of this form. Thus the sets $U(x,F)$
give a countable neighborhood basis at every point and, collectively, a
countable basis for $B$.

We next show, by induction on the rank of $x$, that every $U(x,F)$ is
compact. If $x$ is a leaf, then $U(x,F)=\{x\}$. Suppose that the assertion
has been established below $x$, and let $\mathcal{U}$ be an open cover of $%
U(x,F)$. Choose $V\in \mathcal{U}$ with $x\in V$. There is a finite set $%
E\subseteq \mathrm{Ch}(x)$ such that $\downarrow t\subseteq V$ for every $%
t\in \mathrm{Ch}(x)\setminus E$. The part of $U(x,F)$ not already covered by 
$V$ is therefore contained in the union of the finitely many principal
subtrees $\downarrow t$ with $t\in E\setminus F$. Each of these subtrees is
compact by the inductive hypothesis, since $\downarrow t=U(t,\varnothing )$
and $t$ has rank strictly smaller than $x$. Consequently, finitely many
elements of $\mathcal{U}$, together with $V$, cover $U(x,F)$. This proves
compactness and establishes (1) and (2).

For (3), observe that the principal subtrees below distinct roots are
pairwise disjoint open sets and that 
\begin{equation*}
B=\bigcup \{\downarrow r:r\text{ is a root of }B\}.
\end{equation*}%
Every $\downarrow r=U(r,\varnothing )$ is compact by the preceding argument.
Hence $B$ is compact when it has finitely many roots. If it has infinitely
many roots, the displayed family is an open cover with no finite subcover.

For (4), let $x,y\in B$ be distinct. If they are incomparable, then $%
\downarrow x$ and $\downarrow y$ are disjoint open neighborhoods of $x$ and $%
y$. Suppose instead, after interchanging $x$ and $y$ if necessary, that $x<y$%
. There is a unique child $s$ of $y$ such that $x\leq s$. The sets 
\begin{equation*}
\downarrow x\quad \text{and}\quad U(y,\{s\})
\end{equation*}%
are disjoint open neighborhoods of $x$ and $y$, respectively. Thus $B$ is
Hausdorff.

We now prove (5). Suppose first that $b_{n}\rightarrow b$, and let $%
(b_{n_{k}})$ be a subsequence. Since $\downarrow b$ is an open neighborhood
of $b$, we have $b_{n_{k}}\leq b$ for all sufficiently large $k$, so $b$ is
an eventual upper bound. Let $c$ be any eventual upper bound of $(b_{n_{k}})$%
. For a sufficiently large $k$ we have simultaneously $b_{n_{k}}\leq b$ and $%
b_{n_{k}}\leq c$. The forest property implies that $b$ and $c$ are
comparable. If $c<b$, let $s$ be the child of $b$ such that $c\leq s$. The
neighborhood $U(b,\{s\})$ is disjoint from $\downarrow c$, contradicting
both the convergence of $(b_{n_{k}})$ to $b$ and the fact that its terms are
eventually bounded above by $c$. Therefore $b\leq c$. Thus $b$ is the least
eventual upper bound of every subsequence.

Conversely, suppose that 
\begin{equation*}
b=\limsup_{k}b_{n_{k}}
\end{equation*}%
for every subsequence $(b_{n_{k}})$. Applying the assumption to the original
sequence shows that $b_{n}\leq b$ eventually. Let $O$ be a neighborhood of $%
b $, and choose a finite $F\subseteq \mathrm{Ch}(b)$ such that $%
U(b,F)\subseteq O $. If infinitely many $b_{n}$ were outside $O$, we could
choose a subsequence outside $O$ whose terms are all below $b$. Every term
of this subsequence would lie below a child in $F$. Passing to a further
subsequence, there would be a fixed $s\in F$ above all its terms. Hence $s$
would be an eventual upper bound of that further subsequence. Its least
eventual upper bound is $b$ by hypothesis, so $b\leq s$, contradicting $s<b$%
. It follows that $b_{n}\rightarrow b$.

For (6), if $x\in \sigma _{\alpha }B$, then $\downarrow x$ is an open
neighborhood of $x$. Every strict predecessor of $x$ has rank strictly less
than $\alpha $, and hence $\downarrow x{}\cap \sigma _{\alpha }B=\{x\}$.
Thus $\sigma _{\alpha }B$ is discrete.

Finally, let $Y$ be a nonempty subset of $B$. Since the order is
well-founded, $Y$ has a minimal element $y$. Then 
\begin{equation*}
Y\cap \ \downarrow y=\{y\}.
\end{equation*}%
As $\downarrow y$ is open, $y$ is isolated in $Y$. This proves (7).
\end{proof}

\subsection{Canonical bushes\label{Subsection:bush}}

For every nonzero countable ordinal $\xi $, define a sequence $(\xi
_{n})_{n\in \mathbb{N}}$ of ordinals below $\xi $ as follows. If $\xi $ is a
successor ordinal, set $\xi _{n}=\xi -1$ for every $n\in \mathbb{N}$. If $%
\xi $ is a limit ordinal, let $(\xi _{n})_{n\in \mathbb{N}}$ be a strictly
increasing cofinal sequence of \emph{successor }ordinals in $\xi $.

We define recursively for $\alpha <\omega _{1}$ a bushland $T^{\alpha }$
and, for $1\leq \alpha $, a bushland $I^{\alpha }$ as in \cite[Section 8]%
{lupini_complexity_2025}. Define thus $T^{0}=\left\{ 0\right\} $. Suppose
that $T^{\beta }$ have been defined for $\beta <\alpha $.\ Define $I^{\alpha
}$ to be the union of pairwise incomparable copies of $T^{\alpha _{n}}$ for $%
n\in \mathbb{N}$. We represent the elements of $I^{\alpha }$ as pairs $%
\left( i;n\right) $ for $n\in \mathbb{N}$ and $i\in T^{\alpha _{n}}$. Define
then $T^{\alpha }$ to be $I^{\alpha }\cup \left\{ \alpha \right\} $ with
root $\alpha $.

For $i\in T^{\alpha }$ of positive rank, define the sequence $(i\smallfrown
n)_{n\in \mathbb{N}}$ of its children recursively by setting 
\begin{equation*}
\alpha \smallfrown n=(\alpha _{n};n)
\end{equation*}%
and 
\begin{equation*}
(i;k)\smallfrown n=(i\smallfrown n;k)
\end{equation*}%
for $i\in T^{\alpha _{k}}$ and $n,k<\omega $. Thus the subtree below $\alpha
\smallfrown n$ is isomorphic to $T^{\alpha _{n}}$.

Define $T_{\gamma }^{\alpha }$ to be the set of nodes of $T^{\alpha }$ of
rank equal to $\gamma $. An element of $T_{\gamma }^{\alpha }$ can be
represented as $(k_{1},\ldots ,k_{\ell })$, where $k_{1},\ldots ,k_{\ell
}\in \mathbb{N}$ and $\ell \geq 0$ is the \emph{height}; the empty tuple
represents the root. The \emph{walk} from $\alpha $ associated with this
node is the sequence $(\beta _{0},\ldots ,\beta _{\ell })$ defined by $\beta
_{0}=\alpha $, $\gamma =\beta _{\ell }$, and 
\begin{equation*}
\beta _{i}=\left( \beta _{i-1}\right) _{k_{i}}\qquad (1\leq i\leq \ell )%
\text{.}
\end{equation*}%
Thus nodes of rank $\beta $ in $T^{\alpha }$ can be regarded as walks from $%
\alpha $ to $\beta $. We will also note the following finiteness
consequence. If $v\in T^{\alpha }$ and $\gamma <\mathrm{rk}(v)$, then only
finitely many children $v\smallfrown n$ have rank strictly below $\gamma $.

Every $\beta \leq \alpha $ occurs as the rank of a node of $T^{\alpha }$.
Indeed, starting at a node of rank $\xi >\beta $, pass to a child of rank $%
\xi -1$ when $\xi $ is a successor, and choose a child of rank $\xi _{n}\geq
\beta $ when $\xi $ is a limit. This process terminates after finitely many
steps, since there is no infinite strictly decreasing sequence of ordinals.

\section{Realization\label{Section:realization}}

In this section, we produce examples of phantom Banach spaces and phantom Fr%
\'{e}chet spaces that realize all the possible values of the phantom
spectrum. We continue to assume that $K$ is a Polish non-Archimedean
non-trivially valued field.

\subsection{Spaces}

We begin by recalling some terminology from \cite[Chapter 3]%
{van_rooij_non-archimedean_1978}. For a \emph{set} $X$ one defines $\ell
^{\infty }\left( X\right) $ to be the Banach space of bounded functions $%
X\rightarrow K$ endowed with the norm%
\begin{equation*}
\left\Vert f\right\Vert _{\infty }:=\mathrm{sup}\left\{ \left\vert f\left(
x\right) \right\vert :x\in X\right\} \text{;}
\end{equation*}%
see \cite[Example 3.A]{van_rooij_non-archimedean_1978}. One defines $%
c_{0}\left( X\right) $ to be the closed subspace of $\ell ^{\infty }\left(
X\right) $ comprising the bounded $f:X\rightarrow K$ such that, for every $%
\varepsilon >0$, 
\begin{equation*}
\left\{ x\in X:\left\vert f\left( x\right) \right\vert >\varepsilon \right\}
\end{equation*}%
is \emph{finite}; see \cite[Example 3.B]{van_rooij_non-archimedean_1978}. In
particular, when $X=\mathbb{N}$ one sets $\ell ^{\infty }=\ell ^{\infty
}\left( \mathbb{N}\right) $ and $c_{0}=c_{0}\left( \mathbb{N}\right) $.

More generally, given a separable Banach space $E$ one can define the space $%
\ell ^{\infty }\left( X,E\right) $ of bounded functions $X\rightarrow E$
with the supremum norm. Then $c_{0}\left( X,E\right) $ is the closure of the
subspace of finitely-supported functions.

Notice that a Polish space is zero-dimensional if and only if it has a
compatible ultrametric. Suppose that $X$ is a \emph{locally compact }%
zero-dimensional Polish space, and $E$ is a separable Banach space. One
defines $C_{b}\left( X,E\right) $ to be the space of all \emph{bounded }and 
\emph{continuous }functions $X\rightarrow E$, which is again a closed
subspace of $\ell ^{\infty }\left( X,E\right) $; cf. \cite[Example 3.D]%
{van_rooij_non-archimedean_1978} for the scalar-valued case. One also
defines $C_{0}\left( X,E\right) $ to be the closed subspace of $C_{b}\left(
X,E\right) $ comprising the functions $f:X\rightarrow E$ that \emph{vanish
at infinity}, in the sense that, for every $\varepsilon >0$,%
\begin{equation*}
\left\{ x\in X:\left\Vert f\left( x\right) \right\Vert \geq \varepsilon
\right\}
\end{equation*}%
is compact; cf. \cite[Example 3.F]{van_rooij_non-archimedean_1978} for the
scalar-valued case. When $X$ is a locally compact zero-dimensional Polish
space and $E$ is a separable Banach space, $C_{0}\left( X,E\right) $ is a
separable Banach space. Indeed, fix a countable basis $\mathcal{U}$ of
compact open subsets of $X$ and a countable dense subset $D$ of $E$. The
functions of the form $\sum_{j=1}^{m}\mathbbm{1}_{U_{j}}e_{j}$, where $%
U_{j}\in \mathcal{U}$ and $e_{j}\in D$, form a countable dense subset of $%
C_{0}\left( X,E\right) $. For the scalar-valued compact case, see \cite[%
Exercise 3.T]{van_rooij_non-archimedean_1978}. (Note that we are assuming
that $K$ is Polish, whence a Banach space over $K$ is separable if and only
if it has countable type \cite[pp. 65--66]{van_rooij_non-archimedean_1978}.)
When $X$ is a discrete topological space, $C_{0}\left( X,E\right)
=c_{0}\left( X,E\right) $. When $X$ is a compact zero-dimensional Polish
space, $C_{0}\left( X,E\right) =C_{b}\left( X,E\right) $ is the space of all
continuous functions $X\rightarrow E$. In this case, we denote this space by 
$C\left( X,E\right) $.

If $\left( V_{n}\right) $ is a sequence of Banach spaces, we define%
\begin{equation*}
\bigoplus\nolimits_{n}^{\mathrm{c}_{0}}V_{n}:=\left\{ \left( v_{n}\right)
\in \prod_{n\in \mathbb{N}}V_{n}:\mathrm{lim}_{n}\left\Vert v_{n}\right\Vert
=0\right\}
\end{equation*}%
endowed with the norm%
\begin{equation*}
\left\Vert \left( v_{n}\right) \right\Vert =\mathrm{sup}_{n}\left\Vert
v_{n}\right\Vert \text{.}
\end{equation*}%
For Banach spaces $V$ and $W$, we also let $V\oplus W$ be the Banach space
obtained from the algebraic direct sum endowed with the norm%
\begin{equation*}
\left\Vert v\oplus w\right\Vert :=\max \left\{ \left\Vert v\right\Vert
,\left\Vert w\right\Vert \right\} \text{.}
\end{equation*}

\subsection{Models}

We conclude the proof of Theorem \ref{Theorem:main} by producing spaces with
a Banach cover of any given possible phantom length.

\begin{theorem}
\label{Theorem:bush-length}For every ordinal $\delta <\omega _{1}$, there
exist Banach spaces $X$ and a Banach subspace $N$ such that $X/N$ has
phantom length equal to $\delta +1$.
\end{theorem}

\begin{proof}
The construction is the non-Archimedean version of the construction in \cite[%
Section 10]{lupini_complexity_2025}.

Put $\eta :=1+\delta $. Thus, if $\delta <\omega $, then $\eta =\delta +1$,
while if $\delta $ is infinite, then $\eta =\delta $. Recall from Section~%
\ref{Section:non-Archimedean} that $0<\left\vert \pi \right\vert <1$. For $%
n\in \mathbb{N}$, set%
\begin{equation*}
q_{n}:=\pi ^{-n}\text{.}
\end{equation*}%
Then:

\begin{itemize}
\item $\left\vert q_{n}\right\vert \longrightarrow +\infty $;

\item $\left\vert q_{n}^{-1}\right\vert \leq 1$;

\item $q_{n}^{-1}\longrightarrow 0$.
\end{itemize}

We let $T_{0}^{\eta }$ be the bush defined as in Section \ref%
{Subsection:bush}. Write $T_{\leq \gamma }^{\eta }$ to be the set of nodes
in $T^{\eta }$ of rank at most $\gamma $. For $\gamma \leq \eta $, define
the set of edges crossing the level $\gamma $ by%
\begin{equation*}
\kappa _{\gamma }T^{\eta }:=\{(v,n):\mathrm{rk}(v)>\gamma >\mathrm{rk}%
(v^{\frown }n)\}.
\end{equation*}%
Write $c(\overline{\mathbb{N}},K)$ for the Banach space of convergent
sequences in $K$, with the supremum norm. Set%
\begin{equation*}
X_{0}:=c_{0}(T_{0}^{\eta },K).
\end{equation*}%
For a terminal node $v$, put $\boldsymbol{x}(v)=(x(v))_{n\in \mathbb{N}}$.
Suppose that $0<\gamma \leq \eta $ and that $X_{\sigma }$ and the values $%
x(v)$ at nodes of rank at most $\sigma $ have been defined for $\sigma
<\gamma $. Put%
\begin{equation*}
X_{<\gamma }:=\bigcap_{\sigma <\gamma }X_{\sigma }
\end{equation*}%
with its projective Fr\'{e}chet topology. If $v$ has rank $\gamma $, set%
\begin{equation*}
\boldsymbol{x}(v):=(q_{n}x(v^{\frown }n))_{n\in \mathbb{N}}
\end{equation*}%
whenever this sequence converges, and then define%
\begin{equation*}
x(v):=\lim_{n\rightarrow \infty }q_{n}x(v^{\frown }n).
\end{equation*}%
Define $X_{\gamma }$ to be the set of $x\in X_{<\gamma }$ satisfying all
three conditions below:

\begin{enumerate}
\item $\boldsymbol{x}(v)$ converges for every $v\in T_{\gamma }^{\eta }$;

\item $(\boldsymbol{x}(v))_{v\in T_{\leq \gamma }^{\eta }}$ belongs to $%
c_{0}(T_{\leq \gamma }^{\eta },c(\overline{\mathbb{N}},K))$;

\item $(q_{n}x(v^{\frown }n))_{(v,n)\in \kappa _{\gamma }T^{\eta }}$ belongs
to $c_{0}(\kappa _{\gamma }T^{\eta },K)$.
\end{enumerate}

Endow $X_{\gamma }$ with the norm%
\begin{equation*}
\begin{split}
\Vert x\Vert _{\gamma }:=\max \{& \Vert (\boldsymbol{x}(v))_{v\in T_{\leq
\gamma }^{\eta }}\Vert _{c_{0}(T_{\leq \gamma }^{\eta },c(\overline{\mathbb{N%
}},K))}, \\
& \Vert (q_{n}x(v^{\frown }n))_{(v,n)\in \kappa _{\gamma }T^{\eta }}\Vert
_{c_{0}(\kappa _{\gamma }T^{\eta },K)}\}.
\end{split}%
\end{equation*}

Then one can prove by induction on $\beta \leq \delta $, which guarantees $%
1+\beta \leq 1+\delta =\eta $, that:

\begin{enumerate}
\item $X_{1+\beta }$ is a separable Banach space;

\item the inclusion $X_{1+\beta }\rightarrow X_{0}$ is continuous with dense
image;

\item $X_{1+\beta }$ is a proper dense $\boldsymbol{\Pi }_{3}^{0}$ subspace
of $X_{<1+\beta }$, whence $X_{<1+\beta }/X_{1+\beta }$ has phantom length
equal to $1$ by Lemma \ref{Lemma:length-less-than-1};

\item $\mathrm{Ph}^{\beta }\left( X_{0}/X_{\eta }\right) =X_{<\left( 1+\beta
\right) }/X_{\eta }$.
\end{enumerate}

In particular, this implies that 
\begin{equation*}
\mathrm{Ph}^{\delta }\left( X_{0}/X_{\eta }\right) =X_{<\left( 1+\delta
\right) }/X_{\eta }=X_{<\eta }/X_{\eta }\text{.}
\end{equation*}
Since 
\begin{equation*}
\mathrm{Ph}^{1}\left( X_{<\eta }/X_{\eta }\right) =0
\end{equation*}%
this in turn implies that $X_{0}/X_{\eta }$ has phantom length $\delta +1$,
concluding the proof.
\end{proof}

One can alternatively prove Theorem \ref{Theorem:bush-length} by considering
a non-Archimedean analogue of the spaces constructed by Saint-Raymond in 
\cite[Th\'eor\`eme 31]{saint-raymond_espaces_1976}.

We conclude with the proof of Theorem \ref{Theorem:main}, or its
reformulation as in Theorem \ref{Theorem:spectra}.

\begin{proof}
(1) By Lemma \ref{Lemma:upper-bound}(2), $\sigma _{\mathrm{Ph}}\left( 
\mathbf{Ban}\left( K\right) \right) $ is contained in the set of countable
zero or successor ordinals. The other inclusion follows from Theorem \ref%
{Theorem:bush-length}.

(2) By Lemma \ref{Lemma:upper-bound}(1), $\sigma _{\mathrm{Ph}}\left( 
\mathbf{Fre}\left( K\right) \right) $ is contained in the set of countable
ordinals. If $\lambda $ is a limit ordinal, then one can see that $\lambda
\in \sigma _{\mathrm{Ph}}\left( \mathbf{Fre}\left( K\right) \right) $ by
taking a product of phantom Banach spaces of phantom length $\lambda _{n}$
for $n\in \mathbb{N}$.
\end{proof}

\bibliographystyle{amsplain}
\bibliography{bibliography}

\end{document}